\documentclass[11pt,leqno]{amsart}
\usepackage{amscd}
\usepackage{color}
\usepackage[all,pdf]{xy}
\usepackage{amssymb}
\usepackage{amsfonts}
\usepackage{latexsym}
\usepackage{verbatim}
\usepackage{amsthm}
\theoremstyle{plain}

\newtheorem*{theorem*}{Theorem A}
\newtheorem*{theorem**}{Theorem B}
\theoremstyle{plain}
\newtheorem{theorem}{Theorem}[section]
\newtheorem{corollary}[theorem]{Corollary}

\newtheorem{lemma}[theorem]{Lemma}
\theoremstyle{definition}

\newtheorem{oss}[theorem]{Remark}

\newcommand\R{{\mathbb R}}

\renewcommand{\bar}[1]{\overline{#1}}

\title[Solutions to the Strominger system with torus symmetrry ]{Solutions to the Hull-Strominger system with torus symmetry}

\begin{document}

\author{Anna Fino, Gueo Grantcharov  and Luigi Vezzoni}
\date{\today}
\address{Dipartimento di Matematica \lq\lq Giuseppe Peano\rq\rq \\ Universit\`a di Torino\\
Via Carlo Alberto 10\\
10123 Torino\\ Italy}
 \email{annamaria.fino@unito.it, luigi.vezzoni@unito.it}
 \address{ Department of Mathematics and Statistics Florida International University\\
  Miami Florida, 33199, USA}
\email{grantchg@fiu.edu}

\thanks{The work of the first and third authors was supported by the project FIRB ``Geometria differenziale e teoria geometrica delle funzioni'',
 the project PRIN 2017  \lq \lq Real and Complex Manifolds: Topology, Geometry and Holomorphic Dynamics"
 and by G.N.S.A.G.A. of I.N.d.A.M. The work of the second author was supported by the Simons Foundation grant \#246184}

\maketitle

\begin{abstract} We construct new   smooth solutions to the Hull-Strominger system, showing that the  Fu-Yau solution on  torus bundles over K3 surfaces can be generalized to torus bundles over  K3 orbifolds. In particular, we prove  that, for $13 \leq k \leq 22$ and $14\leq r\leq 22$,  the smooth manifolds  $S^1\times \sharp_k(S^2\times S^3)$ and $\sharp_r  (S^2 \times S^4) \sharp_{r+1} (S^3 \times S^3)$,  have  a complex structure with trivial canonical bundle and admit a solution to the Hull-Strominger  system.

\end{abstract}

\section{Introduction}

The initial proposal for a superstring compactification \cite{CHSW} considered  a 10-dimensional space-time as a metric product of a 4-dimensional maximally supersymmetric space time $N$ and a 6-dimensional compact K\"ahler Calabi-Yau internal manifold $M$. Around the same time A. Strominger \cite{Strominger} and C. Hull \cite{Hull86} considered a heterotic superstring background where $N$  has a wrapping factor. Then the internal space $M$ is still complex with trivial canonical bundle, but   no longer K\"ahler. The supersymmetry conditions lead to what is known as  the  Hull-Strominger system.
To describe it, let $M$ be a compact complex manifold of complex dimension 3 with holomorphically trivial canonical bundle, so that it admits a nowhere vanishing
holomorphic (3,0)-form $\psi$.  Let $V$ be a  smooth   complex vector bundle over $M$ with a Hermitian metric  $H$ along its fibers and let $\alpha' \in \R$ be a constant, also called  the slope parameter.
The  Hull-Strominger system,   for the fundamental form $\omega$  of a Hermitian metric  $g$ on $M$ and a unitary connection $\nabla^H$ on $(V, H)$,  is given by:
\begin{eqnarray}
&& \label{SS1} F_H \wedge \omega^2=0;\\
&& \label{SS2} F_H^{2,0}=F_H^{0,2}=0;\\
&& \label{SS3}i\partial \bar\partial \omega=\frac{\alpha'}{4}\,{\rm tr}\left(R_{\nabla}\wedge R_{\nabla}-F_H \wedge  F_H\right); \\
&& \label{SS4}d(\|\psi\|_\omega\,\omega^2)=0,
\end{eqnarray}
where  $F_H$ and $R_{\nabla}$  are respectively  the curvatures of   $\nabla^H$  and of a metric connection $\nabla$  on  $TM$. Using the equation (\ref{SS2}) one can endow $V$ with a holomorphic structure such that $\nabla^H$ is the Chern connection for $H$. Then (\ref{SS1}) and (\ref{SS2}) describe the Hermitian-Yang-Mills equations for $\nabla^H$. 

The last equation says that $\omega$ is conformally balanced. It was originally written as
$$
\delta \omega=i(\bar \partial-\partial)\ln \|\psi\|_{\omega},
$$
where  $\delta$ is the co-differential $ *^{-1}  d \, *$,  and Li and Yau proved in \cite{LY2005} that it could restated as in \eqref{SS4}.
In the equation $\eqref{SS3}$, known as the Bianchi identity or anomaly cancellation equation, there is an ambiguity in the choice of a metric connection $\nabla$ on $TM$, due to its origins in heterotic string theory  \cite{Hull86, Strominger}. Also from  physical perspective one has $\alpha' \geq 0$ with $\alpha'=0$ corresponding to the K\"ahler case, but in mathematical literature the case $\alpha'<0$ is also considered \cite{PPZ17unim}. Different choices of $\nabla$ and their physical meaning are discussed in \cite{D-S}.
In the present paper  we will  consider the case that $\nabla$ is the Chern connection of $\omega$ and we will denote its curvature  by $R$.

The first solutions of the Hull-Strominger  system on compact non-K\"ahler manifolds, taking $\nabla$  as the Chern connection of $\omega$, were found in  the seminal  work by Fu and Yau \cite{FuYau2007,FuYau}.  The solutions are defined on toric bundles over K3 surfaces.
In  \cite{CE}  Calabi and Eckmann constructed a complex structure on a principal toric bundle over the product $\mathbb{CP}^n \times \mathbb{CP}^m$. The Calabi-Eckmann construction can be easily generalized to any complex base manifold (see e.g. \cite{AGG}) and  Goldstein and Prokushkin showed in \cite{GP} that for a Ricci-flat base and an appropriate choice of the principal torus bundle, the total space has trivial canonical bundle and admits a balanced metric. Starting from the result of Goldstein and Prokushkin, Fu and Yau showed that the Hull-Strominger system on some principal torus fibrations on K3 manifolds can be reduced to a complex Monge-Amp\`ere type equation for a scalar function on the base, and solved it by means of hard analytical techniques (see also  \cite{PPZ16, PPZ18}).

Since then, and the work by Li and Yau  \cite{LY2005}, the successive studies of different analytical and geometrical aspects of the Hull-Strominger system have had an important influence to  non-K\"ahler  complex geometry (see  for instance  \cite{Fei18,Fernandez16, PPZ18}).
Up to now the biggest pool of solutions is provided by the choice  of $\nabla$ given by the Chern connection  \cite{C1, C2, FHP, FHP17,FY14, FIUV, FTY09, OUV17, PPZ17, PPZ16, PPZ17unim,PPZ18}, which includes the first solutions found by Fu, Li, Tseng, and Yau.  More recently, new examples of solutions of the Hull-Strominger system on non-K\"ahler torus bundles over K3 surfaces originally considered by Fu and Yau, with the property that the connection $\nabla$  is Hermitian-Yang-Mills have been constructed in \cite{Fernandez18}. For the physical aspects and significance of various choices for $\nabla$ see \cite{D-S}.

 The theorem of Fu and Yau  in  \cite{FuYau}  states that,   given a compact K3 surface $(S,\omega_S)$ equipped with two anti-self-dual $(1,1)$-forms $\omega_1$ and $\omega_2$  such that $[\omega_1], [\omega_2]\in H^2(S,\mathbb{Z})$
 and with a stable holomorphic vector bundle $E$ of degree $0$ over $(S,\omega_S)$ satisfying
$$
\alpha'(24-(c_2(E)-\frac 1 2 c_1^2(E)))=\frac{1}{4\pi^2}\int_S(\|
\omega_1\|^2+\|\omega_2\|^2)\frac {\omega_S^2}{2},
$$
then there exist a smooth Hermitian manifold $(M,\omega_u = \pi^* (e^u \omega_S) + \frac{i}{2} \theta \wedge \overline \theta)$ and a metric $h$ along the fibers of  $E$ such that $M$ is a principal torus bundle over $S$ and $(V=\pi^*E,H=\pi^*(h), M, \omega_u)$ solves the Hull-Strominger system.  The construction of the torus bundle is  due to Goldstein and  Prokushkin  \cite{GP} and the proof  of Fu and Yau amounts essentially to finding $u$ such that the  condition \eqref{SS3}  holds. The ansatz reduces the anomaly cancelation condition to a scalar equation on $S$ and the formula above is its integrability condition where 24 is the Euler number of $S$. The key point is that this equation can be studied on  the K3 surface $S$ and reduces to a complex Monge-Amp\`ere type equation, which can be solved using a continuity method type argument inspired from the techniques of Yau in \cite{Yau}.  We mention here that \cite{FuYau} provides the only known so far simply-connected compact non-K\"ahler 6-manifold admitting a solution of the Hull-Strominger system.

Our main result claims that the theorem of Fu and Yau generalizes to K3 orbifolds,  extending the result to Hermitian $3$-folds foliated by non-singular elliptic curves. In this way we obtain new simply-connected compact examples carrying solutions of the Hull-Strominger system. A construction of non-K\"ahler Calabi-Yau spaces using an orbifold  base was suggested in \cite[Section 6.3]{GGP}.

\begin{theorem*}\label{main1}
Let $X$ be a compact K3 orbifold  with a Ricci-flat K\"ahler form $\omega_X$  and orbifold Euler number $e(X)$. Let
$\omega_1$ and $\omega_2$ be anti-self-dual $(1,1)$-forms on $X$
such that $[\omega_1], [\omega_2]\in H_{orb}^2(X,\mathbb{Z})$ and the total space $M$ of the principal $T^2$  orbifold bundle
$\pi: M \rightarrow X$ determined by them is smooth. Let $E$ be a stable vector bundle of degree $0$ over $(X,\omega_X)$ such that
\begin{equation}\label{onE}
\alpha'(e(X)-(c_2(E)-\frac 1 2 c_1^2(E)))=\frac{1}{4\pi^2}\int_X(\|
\omega_1\|^2+\|\omega_2\|^2)\frac {\omega_X^2}{2}.
\end{equation}
Then $M$ has a  Hermitian structure $(M,\omega_u)$ and there is a  metric $h$ along the fibers of  $E$ such that  $(V=\pi^*E,H=\pi^*(h), M, \omega_u)$ solves the  Hull-Strominger
system.
\end{theorem*}

The proof of Theorem A  is based on Theorem \ref{toric} in Section 2 and on Theorem \ref{main} in Section 3.  Theorem \ref{toric}
implies the existence of a complex structure on $M$ carrying a balanced metric and a transverse Calabi-Yau structure, while Theorem \ref{main} states the existence of a solution to the Hull-Strominger system on some complex $3$-folds equipped with a Calabi foliation.

Finally we would mention that the topology of compact simply-connected 6-dimensional manifolds have been well studied in the 60's and 70's. In particular there is a  topological classification for compact simply-connected 6-manifolds with  a free $S^1$-action \cite{GL}, which leads to few simple explicit examples. This could be compared to the millions of examples of K\"ahler Calabi-Yau spaces, many of which also carry an elliptic fibrations. Using the classification in \cite{GL} we obtain:

\begin{theorem**}\label{maincor} Let  $13\leq k\leq 22$ and $14\leq r \leq 22$. Then on the smooth manifolds   $S^1\times \sharp_k(S^2\times S^3)$ and $\sharp_r (S^2\times S^4)\sharp_{r+1}(S^3\times S^3)$ there are  complex structures with trivial canonical bundle admitting a balanced metric and a solution to the Hull-Strominger system via the Fu-Yau ansatz.
\end{theorem**}

 The cases $k=22$ and $r=22$ respectively correspond to the  solutions of Fu and Yau.  The examples in Theorem B    have  the structure of a principal $S^1$-bundle over  Seifert $S^1$-bundles and for them $\alpha '>0$.
 The   simply-connected examples are obtained  starting from a Calabi-Yau orbifold surface (K3 orbifold) with isolated $A_1$ singular points and trivial orbifold fundamental group. Then the construction uses partial resolution of singularities.

 We describe shortly the structure of the paper. In Section 2 we collect the necessary information on orbifolds. We focus on the case in which the singular points are isolated and the local holonomy groups are cyclic, although many results could be generalized for other groups. We collect the sufficient conditions we need for smoothness of the appropriate Seifert $S^1$-bundles over simple orbifolds in Theorem \ref{cororb}. The geometric properties like existence of connections with prescribed curvature on toric bundles is in Theorem \ref{toric}. We consider $T^2$-bundles as sequence of $S^1$-bundles and prove the main topological facts in Corollary \ref{cormain}. In Section 3 we show that the Fu-Yau proof on the existence of a solution to the Hull-Strominger system on principal $T^2$-fibrations over K3 surfaces generalizes to some compact foliated Hermitian $3$-folds equipped with a  transverse Calabi-Yau structure. This section is the analytic core of the paper and makes use of the continuity method taking into account some hard a priori estimates in \cite{FuYau,picard} and a theorem of El Kacimi in \cite{EKA}. Finally combining the results of the previous sections, in Section 4 we prove Theorem B.

\section{ $T^2$-bundles over orbifolds} \label{sect2}

Recall that an orbifold is  a space covered by  charts which are homeomorphic maps into some quotients of open sets of an Euclidean space modulo finite groups. The singular points are the points in which the isotropy subgroup is non-trivial. Classical examples of orbifolds are the leaf spaces of Riemannian foliations with compact leaves (see e.g. \cite{molino}).

In this paper we will consider   complex  orbifolds, where the chart transitions are holomorphic (see the definitions in \cite[Ch. 4]{BG}) and  restrict ourselves to the case in which the singular points are isolated and the local holonomy groups are cyclic. In particular,  we are interested in  special foliations on complex manifolds in which the leaves are elliptic curves and have a structure of a principal bundle over the leaf space, which we  will call  {\em $T^2$-bundles} and the ways to construct them from an orbifold leaf space which is  a complex surface.

We will need the standard notions of vector bundles (or V-bundles),  tensors and sheaves, as well as the basic topological invariants transferred to the orbifold case.
As in the smooth case every  complex orbifold  has a Hermitian metric  and  every principal bundle over an orbifold admits a connection (for a proof see for instance  \cite[Theorem 3.16]{LTX}). The Chern-Weil theory proceeds on orbifolds as well, so  the Chern classes of a holomorphic orbifold bundle  are defined in terms of Hermitian (orbifold) metrics and their Chern curvature. If  we  identify divisors and their Poincar\'e dual 2-forms as in the smooth case, then an ample divisor on a complex orbifold  is represented by a K\"ahler form. We need also the fact that every pair of forms in the characteristic class of a $T^2$-bundle is a curvature of some pair of connection 1-forms. We note here for an orbifold $X$, the orbifold rational cohomology satisfy $H^k_{orb}(X, \mathbb{Q}) = H^k(X, \mathbb{Q})$,  so $H^k_{orb}(X, \mathbb{Z}) \subset H^k(X, \mathbb{Q})$ - a fact we are going to use later.
From  \cite[Theorem 4.3.16]{BG} we get the following standard fact which we formulate explicitly:

\begin{theorem}\label{conectionforms}
For every $T^2$-bundle $\pi: M \rightarrow X$  over an orbifold $X$ with a characteristic class $\alpha,\beta \in  H^2_{orb}(X, \mathbb{Z})$ and any $(\omega_1,\omega_2)\in (\alpha,\beta)$   there  exist  connection 1-forms $\theta_1,\theta_2$ on $M$  such that  $d\theta_i = \pi^*(\omega_i)$.

\end{theorem}

\begin{proof}
The proof is the same as for the smooth case. By the existence result mentioned above  we have connection 1-forms $\theta_i$ on $M$ with curvature  2-forms on $X$ $d\theta_i = \pi^*(\tilde{\omega}_i)$ where by the Chern-Weil isomorphism $[\tilde{\omega_i}] \in (\alpha,\beta)$. Then  $\tilde{\omega_i} = \omega_i+d \alpha_i$ and  the 1-forms $\theta_i+\pi^*(\alpha_i)$ define the required 1-forms. They clearly are closed and in the class $(\alpha,\beta)$. To see that they define connections one has to check that they are equivariant. But for Abelian groups, the equivariant condition is just $T^2$-invariant one.
\end{proof}
%%%%%%%%%%%%%%%%%%
Recall that a holomorphic vector bundle $V$ over a K\"ahler manifold $(X,\omega)$ is called {\it stable} with respect to  the K\"ahler form $\omega$ if,  for every proper coherent subsheaf $E$ of the sheaf of local holomorphic sections of $V$,  the inequality   $deg(E)/rk(E) < deg(V)/rk(V)$ holds, where the degree  is calculated with respect to $\omega$.  The degree of a bundle is the product of its first Chern class with the appropriate power of the cohomology class of $\omega$, so the notion of stability may depend on the choice of   the cohomology class $[\omega^{n-1}]$.
The same definition works in the orbifold case.

In view of Donaldson-Uhlenbeck-Yau Theorem \cite{Donaldson, UY} an irreducible Hermitian vector bundle  $V$ of degree 0 over a K\"ahler manifold $(X,\omega)$ is stable if and only if $V$ has a  Hermitian-Yang-Mills metric $H$, i.e. a Hermitian metric $H$ along the fibers of $V$ whose curvature $F_H$ satisfies $F_H\wedge\omega^{n-1} = F_H^{2,0}=F_H^{0,2} = 0$, where $n$ is the complex dimension of $X$.

This result can be  extended  in two directions. First, the notion of degree of a vector bundle can be defined with respect to some non-K\"ahler metrics of special type, called {\it Gauduchon} metrics, i.e. an  Hermitian  metrics  $\omega$  satisfying  the condition $ \partial \overline \partial \omega^{n-1} = 0$  \cite{LiYau}.  Since a  balanced metric is  a Hermitian metric for which $d\omega^{n-1}=0$,  every balanced metric is also Gauduchon. When the degree of the bundle $V$ is 0, the Hermitian-Yang-Mills condition is conformally invariant,
so the stability is well defined for conformally balanced metrics. The second direction is that $X$ could be a K\"ahler orbifold with \lq \lq nice" singularities. For instance, C. Simpson proved in  \cite{Simpson} the Donaldson-Uhlenbeck-Yau Theorem on certain non-compact K\"ahler manifolds with an appropriate condition at infinity. From his result follows that the Donaldson-Uhlenbeck-Yau theorem is true for  K\"ahler orbifolds with isolated cyclic singularities, since the manifold obtained after deleting the singularities satisfies Simpson's condition  \cite{Simpson}.

The following statement is a combination of several known results for manifolds, which we formulate for orbifolds and  we will use in Sections 3 and 4.

\begin{theorem}\label{toric}
Let $(X, \omega_X)$ be a  K3-orbifold (or Calabi-Yau 2-fold) with  a smooth K\"ahler form $\omega_X$. Let $\omega_1,\omega_2$ be rational cohomology $(1,1)$-classes. Assume that the $T^2$-bundle $\pi: M\rightarrow X$ defined by $c_1(M/X)=[\omega_1,\omega_2]$ is smooth. If $\theta_i$ are  the connection 1-forms with $d\theta_i = \pi^*\omega_i$, then $M$ admits a complex structure such that $\theta = \theta_1+i\theta_2$ is (1,0)-form and $\pi$ is a holomorphic projection. For the complex manifold $M$ the following properties  hold:

i) The Hermitian metric on $M$ defined by  the $(1,1)$-form $\omega = \pi^*(\omega_X) + \theta_1\wedge\theta_2$ is balanced (i.e. $d\omega^2 = 0$) iff $tr_{\omega_X}\omega_1 = tr_{\omega_X}\omega_2 = 0$. If we choose $\omega_1,\omega_2$ to be harmonic, then this is equivalent to the topological condition $[\omega_X]\cup[\omega_1] =[\omega_X]\cup[\omega_2] =0$.

ii) If $\psi_X$ is a holomorphic $(2,0)$-form on $X$ with $||\psi||_{\omega_X} = const$, then the form $\psi = \psi_X\wedge \theta$ is  holomorphic with constant norm with respect to $\omega$.

iii) For every smooth function $u$ on $X$, the metric $\omega_u = e^u\pi^*(\omega_X)+\theta_1\wedge\theta_2$  on $M$  is conformally balanced with conformal factor $||\psi||_{\omega_u}$.

iv) If $E$ is a stable bundle on $X$ with respect to $\omega_X$ of degree 0 and Hermitian-Yang-Mills metric $h$ and curvature $F_h$, then $V=\pi^*(E)$ is a  stable bundle of degree 0 on $M$ with respect to $\omega_u$ with Hermitian-Yang-Mills metric $H=\pi^*(h)$ and curvature $F_H:=\pi^*(F_h)$.
\end{theorem}

\begin{proof}
The proof is the same as in the smooth case. Part $i)$ and $ii)$ are in \cite{GP} (see also \cite{GGP}). Parts $iii)$ and $iv)$ are in \cite{LiYau, FuYau}.
\end{proof}

The next results of this sections will be used in Section 4 to construct explicit examples and to prove Theorem B.
Here we consider $T^2$-bundles over an orbifold $X$ which are given by the following sequence
$$
\xymatrix{
S^1 \ar@{^{(}->} [r] &M \ar@<-0.4ex>[d] \\
S^1 \ar@{^{(}->} [r] &M_1 \ar@<-0.4ex>[d] \\
& X\,
}
$$
where $M_1\to X$ is  a Seifert $S^1$-bundle, $M_1$ is smooth and $M\to M_1$ is a regular principal $S^1$-bundle over $M_1$.
For the definition of {\em Seifert $S^1$-bundle} we refer to \cite[Definition 4.7.6]{BG}. Roughly speaking,  Seifert  $S^1$-bundles  are  spaces  with  a locally free $S^1$-action, for which the $S^1$-foliation has an orbifold leaf space.  A  multiple leave is an $S^1$-orbit on which the action is not globally free.

The following is Theorem 4.7.3 in \cite{BG}, which is proven by  Koll\'ar in \cite{Kollar2004}. For convenience we use  \lq \lq divisors" to refer to rational or integral Cartier divisors to distinguish them from the Weil divisors.
In this terminology an ample rational divisor has as a  Poincar\'e  dual a rational K\"ahler form and vice versa.
\begin{theorem}
Let $X$ be a normal reduced complex space with at worst quotient singularities and $\Delta=\sum_i (1-\frac{1}{m_i})D_i$
be a $\mathbb{Q}$ divisor (this is the data associated to an orbifold). Then there is a one-to-one correspondence between Seifert ${\mathbb C}^*$-bundles $f:Y\rightarrow (X,\Delta)$ and the  following data:

(i) For each $D_i$ an integer $0\leq b_i< m_i$ relatively prime to $m_i$, and

(ii) a linear equivalence class of Weil divisors $B\in Div(X)$.
\end{theorem}

For the purposes of the present paper,  we focus on the case when $Y$ is smooth and we need  to consider the \lq\lq smoothness part\rq\rq \, of Theorem 4.7.7 in \cite{BG}:

\begin{theorem}\label{localseifertsmooth}

If $(X,\Delta)$ is a locally cyclic orbifold as in the Theorem above and $f: Y \rightarrow (X,\Delta)$  is an $S^1$-orbibundle whose local uniformizing groups inject into the group $S^1$ of the orbibundle, then $f: Y \rightarrow (X,\Delta)$ is a Seifert $S^1$-bundle and $Y$  is smooth.

\end{theorem}

For algebraic orbifolds this could be refined (see  Theorem 4.7.8 in \cite{BG} and \cite{Kollar2005}) and for an orbifold $(X,\Delta)$ with trivial $H^1_{orb}(X,{\mathbb Z})$ the Seifert $S^1$-bundle  $Y$  is uniquely determined by its first Chern class $c_1(Y/X)\in H^2(X,{\mathbb Q})$, which is defined as $[B] + \sum_i \frac{b_i}{m_i} [D_i]$.

%%%%%%%%%%%%%%%%%%%%%%%%%%%%%%%%

Next we focus on  $T^2$-bundles over an orbifold surface  $X$. The  general theory of such spaces from the foliations view-point is given in \cite{HS}. In particular, such bundles are determined by two rational divisors on the base orbifold.
%%%%%%%%%%%%%%%%%%%%%%%%%%

We want to use the well-known examples of K3 orbifolds (or Calabi-Yau surfaces) to construct examples carrying solutions of the Hull-Strominger system. Since the examples are  generic hypersurfaces or   complete intersections in weighted projective spaces, we need a criteria for smoothness of the toric bundles over them. To do this we use an indirect approach - instead of using explicit equations we use the fact that the links which these surfaces define are already smooth  Seifert $S^1$-bundles. The following statement could be generalized in many directions, but we need it in this form to construct examples (note that the space $M$ is auxiliary,  it is not used in the constructions later):

\begin{theorem} \label{cororb}

Let $X$ be a   compact complex 2-dimensional   orbifold with only isolated  $A_1$-singularities  and an ample divisor $\bar{H}\in Pic(X)$. Assume that the $S^1$-Seifert bundle $Y \rightarrow X$ defined by $c_1 (Y/X) = H = \frac{p}{q}\bar{H}$ with $\frac{p}{q}>0$ is smooth. Then the  blow-up $\tilde{X}$  of $X$ at any number $k\geq 1$ of singular points has a rational  divisor $D$ and a real positive $(1,1)$-class $\omega$ with the following properties:

i)  $D.\omega =0$ and we say that $D$ is \lq \lq traceless" with respect to $\omega$ in analogy with the smooth case.

ii) If  the canonical bundle of $X$ is trivial, i.e. $K_X=0$, and $\pi_1^{orb}(X)=1$, then $K_{\tilde{X}} = 0$ and $\pi^{orb}_1 (\tilde{X})=1$.

iii) The $S^1$-Seifert bundle $\tilde{Y} \rightarrow \tilde{X}$ determined by  the  cohomology class  $c_1 (\tilde{Y}/\tilde{X}) = D$ is smooth.

iv) The pullback of $\omega$ to  $\tilde{Y}$ is smooth.

\end{theorem}

\begin{proof}
Since all isolated singularities of type  $A_1$  give rise to a crepant resolution, we have that  the canonical bundle triviality is preserved.   Moreover, by \cite[Thm. 7.8]{Kollar}   or  \cite[Thm. 4.1]{Verbitsky}  we  get $\pi_1^{orb}(\tilde{X})=1$.

In case of blow-up of  $A_1$ singularity at a singular point, the exceptional divisor $E$ is a rational curve with  self-intersection $-2$, i.e. $E.E=-2$. In particular the Picard group of $\tilde{X}$ is generated by the pull-back of $H$ and the rational curves $E_1,..,E_k$,  if $\tilde{X}$ is obtained by blowing up $k$ singular points. There are also the relations $H.H>0$, $E_i.E_i=-2$, $H.E_i=0$ and $E_i.E_j = 0$,  for every $i,j=1,2,..,k$, $i\neq j$.

In general,  by the Nakai-Moishezon criteria  a (1,1)-class  associated to a divisor $ L$ on surface  is positive  if the  $ L. L>0$ and
$ L.C>0$,  for every curve  $C$. The criteria is valid for singular surfaces and real (1,1)-classes, see e.g.  \cite{CP}.  Since $H$ is positive (ample), by this criteria we can find $n$ large enough such that $nH-E_1-E_2-\ldots -E_k$ is positive. This follows from the fact that every curve in $\tilde{X}$ is either a pull-back of a curve in $X$ or belongs to the exceptional divisor. Note that $n$ does not have to be integer but only a rational number. Since $H$ is rational and ample, its  pull-back to $\tilde{X}$ defines a smooth (in orbifold sense) (1,1)-form. Combining it with the fact that the exceptional curves define divisor with a well defined smooth first Chern class, $nH-E_1-...-E_k$ provides a positive smooth form $\omega$ on $\tilde{X}$. For the primitive divisor $D$ we can choose $D=H+mE_1$.
We need to prove that there exists an integer $m>0$ such that  the intersection number $(H + m E_1) . (n H- E_1-\ldots -E_k)$ vanishes. But from the intersection numbers above we have
$$
(H+mE_1).(n H-E_1-...-E_k) =  n H.H- mE_1.E_1 = n (H.H) - 2m
$$
and by appropriate choice of $n$ and $m$ which are also large enough, we can have the last number equal zero.

Now we have to check that $D$ defines a smooth  $S^1$-Seifert bundle $ \tilde{Y} \rightarrow \tilde{X}$. This follows from the assumption that $H$ defines a smooth such bundle and according to Theorem \ref{localseifertsmooth} and Proposition 4.7.8 in \cite{BG} all other conditions are about local divisibility of $D$ at the singular points.  But the remaining singular points on $\tilde{X}$ satisfy the conditions for the pull-back of $H$, since they are pull-backs of singular points on $X$ and $H$ satisfies these local conditions since by assumption $Y$  is smooth. Now $E_1,\ldots ,E_k$ do not intersect with the other singularities since they are isolated. So $D$ defines a smooth Seifert $S^1$-bundle $\tilde{Y}$. The last thing to check is that $\omega$ pulls-back to a smooth form on $\tilde{M}$. But since $\omega$ is smooth in the orbifold sense on $\tilde{X}$, then its pull-back is also smooth.
\end{proof}

In dimension 5 and 6 there are strong classification results about the topology of simply-connected compact smooth manifolds admitting  a free $S^1$-action. To have  a nice expression of the diffeomorphic type of these spaces one needs also the (co)homology groups to be torsion-less.
In general the torsion of the cohomology groups of an orbifold surface and their associated Seifert  $S^1$-bundles is a delicate question (see \cite{Kollar2005}). However we need a Calabi-Yau surface and,  for this type of surface,   Proposition  10.2 and Corollary 10.4  in \cite{Kollar2005}  provide  a simple condition to identify the Seifert $S^1$-bundle up to a diffeomorphism. We formulate it here in the form  which is relevant to our purpose:
\begin{theorem}\label{Kol} $($\cite{Kollar2005}$)$
Let $(X,\Delta)$ be a Calabi-Yau orbifold surface, i.e.  $K_X + \Delta$
is numerically trivial, and $\pi: Y \rightarrow (X,\Delta)$ be a Seifert $S^1$-bundle with $Y$ smooth and $H_1(Y,\mathbb{Z})=0$. Then $\Delta \equiv 0$, $X$  is a Calabi-Yau orbifold with trivial canonical class, and $Y$ has no torsion in the cohomology and is diffeomorphic to a connected sum of $k$ copies of $S^2\times S^3$, where $k\leq 21$ is equal to  $k = dim(H^2(X,\mathbb{Q}))-1$.
\end{theorem}
\begin{proof} This is Corollary 10.4 (1) and (3)  in \cite{Kollar2005} and the expression for $k$ follows from Theorem 5.7 in \cite{Kollar2005}. Indeed,  the minimal resolution of $X$  is a K3-surface and $H_2(Y, {\mathbb Z})$  is torsion free of rank at most 22.
\end{proof}
We should note that $H_1(Y,\mathbb{Z})=0$ is essential and for the examples we have in mind it requires that the orbifold fundamental group of $X$ is trivial. So we cannot use the Kummer surface construction, which provides a smooth K3 surfaces after blowing up 16 $A_1$-singularities - in this case $Y$  is not simply-connected. Now we are ready for the main topological characterization of the simply-connected examples we consider in the next sections.
\begin{corollary}\label{cormain}
Let $X$ be a Calabi-Yau orbifold surface (K3 orbifold) with  only isolated $A_1$ singularities  and trivial orbifold fundamental group,  and  $\tilde{X}$ be the blow-up of $X$ at $k\geq 2$ of these points. Assume that there is an ample rational divisor $H$ on $X$, such that the Seifert $S^1$-bundle $Y$ with $c_1(Y/X)=H$ is smooth. Then there exist  two  rational divisors $D_1$ and $D_2$  and a real positive $(1,1)$-class $\omega$ which satisfy (i)-(iv) of Theorem \ref{cororb} for the respective Seifert $S^1$-bundles $\tilde{M_1}$ and $\tilde{M_2}$ over $\tilde{X}$ defined by $D_1$ and $D_2$. The divisors $D_1$ and $D_2$ can be chosen to be independent over $\mathbb{Q}$  and in such a way, that the corresponding $T^2$ - bundle  $\tilde M$
is simply-connected and has no torsion in the cohomology. It is diffeomorphic to $\#_r(S^2\times S^4)\#_{r+1}(S^3\times S^3)$, where $r=rk(H^2(X,\mathbb{Q}) -2$.
\end{corollary}
\begin{proof}
 Choose $D_1=H-mE_1-E_2$ and $D_2 = H - E_1 - mE_2$ and $\omega = nH- E_1-E_2 -\ldots -  E_k$ (as in the proof of Theorem \ref{cororb}). Then $D_1$ and $D_2$ clearly satisfy $i)$ to $iv)$ with this $\omega$ and are independent over $\mathbb Q$. To see that they determine a simply-connected 6-manifold, we consider first the Seifert $S^1$-bundle $\tilde{M_1} \rightarrow \tilde{X}$ corresponding to $c_1(\tilde{M_1}/X)=D_1$. According to Corollary 10 in  \cite{Kollar2007}  $\tilde M_1$ is simply-connected, and by the previous Theorem diffeomorphic to a connected sum of $k$ copies of  $S^2\times S^3$. Now the total space of the $T^2$-bundle  $\tilde M$ can be considered as a smooth $S^1$-bundle $\pi_2: \tilde{M}\rightarrow \tilde{M_1}$. Its characteristic class is the pull-back of $D_2$ to $\tilde M_1$. We can always divide it by an integer, if it is not primitive.
  In this way we will  obtain a simply-connected 6-manifold $\tilde M$.
  To see that it has no torsion in the cohomology
 we can use the proofs in Section 2 of \cite{Kollar2007}.  Although the results there are for bundles over  a base of complex dimension 2, it is easy to generalize the proof for any dimension in the smooth case - see Proposition 14 and Corollary 15 and the remaining parts of the proof of Proposition 8 in \cite{Kollar2007}.
Then $\tilde M$  is a simply-connected 6-manifold with a  free $S^1$-action and vanishing second Stiefel-Whitney class (since $\tilde M$ is a complex manifold and its canonical bundle is trivial). Its cohomology
 can be computed as in the prof of Proposition 8 in \cite{Kollar2007} and in particular have  no torsion. The main result in \cite{GL}  gives the diffeomorphism  type of $\tilde M$  as described. The number $r$ can be computed as $r=rk(H^2(\tilde M_1,\mathbb{Q}))-1 = rk(H^2(X,\mathbb{Q})) -2$ using the exact cohomology sequence for $S^1$-bundles.
\end{proof}
We'll use the  previous results in Section 4 to prove Theorem B.

\section{Proof of Theorem A}
We obtain  the proof of Theorem A  by mixing the results in  Section 2 with the following generalization of the Fu-Yau theorem to Hermitian $3$-folds with a transverse Calabi-Yau structure.

\begin{theorem}\label{main}
Let $(M,I,g)$ be a Hermitian $3$-fold with fundamental form $\omega(\cdot,\cdot)=g(I\cdot, \cdot )$ and equipped with  a closed $(2,0)$-form $\psi_B$
and  a $\partial$-closed $(1,0)$-form $\theta$
 satisfying,
\begin{equation}\label{top}
\|\psi_B\|=\|\theta\|=1\,,\quad d\omega_B=0\,, \quad \omega_B\wedge d\theta = 0\,, \quad  \iota_{\bar Z}d\theta=0\,,\quad   \iota_Z\psi_B=0\,,
\end{equation}
where $\omega_B:=\omega-\frac{i}{2}\theta\wedge\bar\theta$ and $Z$ is the vector field dual to $\theta$ with respect to $g$.
Assume that there exists a Hermitian vector bundle $(V,H)$ over $M$ such that the curvature $F_H$ of its Chern connection satisfies
\begin{equation}\label{f1}
F_H\wedge\omega_B=0\,,\quad F_H^{2,0}=F_H^{0,2}=0\,,
\end{equation}
and
\begin{equation}\label{f2}
\int_M \alpha'{\rm tr}\,(R\wedge R)\wedge \omega-\alpha'{\rm tr}\,(F_H\wedge F_H)\wedge \omega- \|d\theta\|^2\,\omega^3=0\,,
\end{equation}
where $R$ is the Chern connection of $g$ and $\alpha'\in \mathbb R$.
There exists a smooth function $u\colon M\to \R$ such that if $\omega_u={\rm e}^u \omega_B+\frac{i}{2}\theta\wedge \bar \theta$, then $(V,H, M, \omega_u)$ is a solution to Hull-Strominger's system.
 \end{theorem}

We show Theorem \ref{main} by using some results in \cite{FuYau,picard}. Namely, it can be observed that the argument in \cite{FuYau} to reduce the Hull-Strominger  system on a principal  $T^2$-bundle over a K3 surface to an elliptic equation, depends only on the foliated structure of the manifold and the assumptions in  Theorem \ref{main} allow to reduce the Strominger system on $M$ to a transversally elliptic equation. The solvability of the equation follows by \cite{FuYau,picard} taking into account a result  of El Kacimi in \cite{EKA}.

\medskip Now we focus on the setting of Theorem \ref{main}. First, we'll construct a $(3,0)$-form $\psi$  and show that the equations \eqref{SS1}, \eqref{SS2} and \eqref{SS4} are satisfies for any basic function $u$.  Let $Z$ be the dual vector field of $\theta$ and let $\mathcal F$ be the foliation generated by the real and the imaginary part of $Z$. Denote by $L$ the vector bundle of tangent vectors to $M$ orthogonal to $\mathcal F$. The metric $g$ restricts to a Hermitian metric $g_B$ along the fibers of $L$. Note that $\omega_B$ is the fundamental form of $g_B$ since if  $g(W,\bar Z)=0$, then $\omega(W,\bar W)=\omega_B(W,\bar W)$. In particular $\omega_B$ is non-degenerate in the transverse directions with respect to $\mathcal F$.

The pair $(\omega_B,\psi_B)$ induces a transverse Calabi-Yau structure on $M$. From the fact that $M$ is Hermitian 3-fold and  \eqref{top} it follows that  $\omega_B$ and $\psi_B$ are both basic forms, i.e. they both satisfy the conditions
 $$
\iota_X \omega_B =0\,,\quad    \mathcal L_X \omega_B=0\,,\quad   \, \iota_X \psi_B =0\,,\quad  \mathcal L_X \psi_B=0\,,
$$
for every vector field $X$ tangent to the foliation.  We denote by $\Omega_B^r(M)$ the space of complex basic $r$-forms on $M$ and by $*_B\colon \Omega_B^r(M)\to \Omega_B^{4-r}(M)$ the {\em basic Hodge    \lq\lq star\rq\rq \, operator}  which is defined by the usual relation
$$
\eta \wedge *_B \beta=g_B(\eta,\beta)\, \frac{\omega_B^2}{2}\,,
$$
where $g_B$ is extended in $\Omega_B^r(M)$ in the standard way.
From our assumptions it follows that
$d\theta$ is a basic form and
 $$*_Bd\theta=-d\theta.$$

Notice that if  $\psi=\psi_B\wedge \theta$, then the pair $(\omega, \psi)$ gives an ${\rm SU}(3)$-structure on $M$. Furthermore from our assumption $d\theta\wedge \omega_B=0$ we infer
$$
d\omega^2=d\left(\omega_B^2+i\omega_B\wedge\theta\wedge\bar \theta  \right)=i\omega_B\wedge d\theta\wedge \bar \theta-i\omega_B\wedge\theta \wedge d \bar \theta=0\,,
$$
i.e., $\omega$  defines a  balanced  metric.

In order to simplify the notations we put
$$
\chi=\frac{i}{2}\theta\wedge \bar \theta
$$
and following the  Fu-Yau approach, we deform $\omega$ as
$$
\omega_u={\rm e}^u\,\omega_B+\chi
$$
where $u$ is a basic function on $M$ (i.e. $X(u)=0$ for every vector field  $X$ tangent to the foliation). Let $\psi=\psi_B\wedge \theta$. 

Since $\|\psi\|_{\omega_u}={\rm e}^{-u}$, taking into account that $\omega_B$ is closed and $u$ is basic, we have
$$
d(\|\psi\|_{\omega_u}\,\omega_u^2)=d({\rm e}^{-u}\,\left({\rm e}^{2u}\omega_B^2+i{\rm e}^{u}\omega_B\wedge\theta\wedge \bar \theta  \right))=d({\rm e}^{u}\omega_B^2+i\omega_B\wedge\theta\wedge \bar \theta)=0\,,
$$
i.e. $(\omega_u,\psi)$ satisfy \eqref{SS4}. Moreover condition \eqref{f1} implies that $(F_H,\omega_u)$ satisfies the equations  \eqref{SS1} and \eqref{SS2}. Then to solve the  Hull-Strominger system we need to solve equation \eqref{SS3} which reduces to the following equation on $u$
\begin{equation}\label{ppo}
i\partial \bar\partial \omega_u=\frac{\alpha'}{4}\,{\rm tr}\left(R_u\wedge R_u-F_H\wedge F_H\right)\,,
\end{equation}
where $R_u$ is the Chern curvature of $\omega_u$. To complete the proof of Theorem \ref{main} it remains to show the existence of solution of \eqref{ppo}.

Taking into account that $*_B\bar\partial \theta=-\bar\partial \theta$, we get
\begin{equation}\label{pbpou}
i\partial \bar\partial \omega_u=\,i\partial \bar\partial {\rm e}^u\wedge\omega_B+\frac{1}{2} \|\partial \bar \theta\|^2\,\frac{\omega_B^2}{2}\,.
\end{equation}
Moreover, we have the following

\begin{lemma}\label{RuRu}
The following formula holds
$$
{\rm tr} (R_u\wedge R_u) ={\rm tr} (R_B\wedge R_B) +2\partial \bar\partial u\wedge \partial \bar\partial u+2  i\partial \bar\partial({\rm e}^{-u}\rho)
$$
where $R_B$ is the Chern curvature of the Hermitian bundle $(L,g_B)$ and  $\rho$ is a real basic $(1,1)$-form which does not depend on $u$.
\end{lemma}

\begin{proof}
Let $\tilde g_B$ be any Hermitian metric along the fibers of $L$ satisfying $\mathcal L_X  \tilde g_B =0$  and let
$\tilde g=\tilde g_B+\theta\odot\bar \theta$ be the induced metric on $M$. Denote by $\tilde R_B$ and $\tilde R$ the Chern curvature of $\tilde g_B$ and $\tilde g$, respectively. Then we can locally write
$$
\tilde R=\bar\partial (\partial G\cdot G^{-1})
$$
where $G=(\tilde g(Z_i,\bar Z_j))$ and $\{Z_1,Z_2,Z_3\}$ is a local holomorphic frame on $M$.

Let $\{z_1,z_2,z_3\}$ be  complex coordinates on $M$ such that $\partial_{z_3}=Z$. Then
$$
Z_1=\partial_{z_1}-\theta(\partial_{z_1})Z\,,\quad Z_2=\partial_{z_2}-\theta(\partial_{z_2})Z
$$
is a local frame of $L$. In order to compute $\tilde R$, we modified the frame $\{Z_1,Z_2,Z\}$ to a holomorphic  frame. The dual frame to $\{Z_1,Z_2,Z\}$ is given by $\{dz_1,dz_2,\theta\}$ and since $\bar\partial\theta$ is basic and $\bar\partial$-closed,
it can be locally written as $\bar\partial \alpha$, where
$\alpha$ is a basic form.
Hence  $\{dz_1,dz_2,\theta-\alpha\}$ is a local frame of holomorphic forms. The dual frame is
$\{Z_1-\alpha_1Z,Z_2-\alpha_2Z,Z\}$, where $\alpha_k=\alpha(\partial_{z_k})$. With respect to this last frame the matrix $G$ takes the following expression
\begin{equation}\label{dieci}
G=\left(\begin{array}{ccc} g_{1\bar 1}+|\alpha_1|^2 &
g_{1\bar2}+\alpha_1\overline \alpha_2& \alpha_1\\
g_{2\bar1}+\alpha_2\overline \alpha_1 & g_{2\bar2}+|\alpha_2|^2 & \alpha_2\\
\overline\alpha_1& \overline \alpha_2 &
1\end{array}\right)
\end{equation}
where $g_{i\bar j}=\tilde g(Z_i,\bar Z_j)$. Accordingly  we can write \eqref{dieci} as
$$
G=\left(
\begin{array}{cc}
G_B+A\cdot A^* & A\\
A^* & 1 \end{array}
\right),
$$
where $G_B=(g_{i\bar j})$, $i,j=1,2$, and $A=(\alpha_1,\alpha_2)^t$. Now
$$
\tilde R=\bar\partial(\partial G\cdot G^{-1})=(\bar\partial\partial G)\cdot G^{-1}-\partial G\wedge \bar \partial G^{-1}
=(\bar\partial\partial G)\cdot G^{-1}+\partial G\wedge\left( G^{-1}\cdot  \bar \partial G\cdot G^{-1}\right)
$$
and we write
$$
\tilde R=\left(
\begin{array}{cc}
R_{1\bar 1} & R_{1\bar 2}\\
R_{2\bar 1} & R_{2\bar 2} \end{array}
\right)
$$
where $R_{1\bar 1}$ is a $2\times 2$ matrix of $2$-forms  and $R_{2\bar 2}$ is a $2$-form.
At a fixed point $p$ we may assume $A(p)=0$ and change the coordinates in order that the matrix $G_B$ (and the matrix $G$) are the identity at $p$.
Therefore
$$
(\bar\partial\partial G)\cdot G^{-1}=
\left(
\begin{array}{cc}
\bar\partial\partial(G_B+A\cdot A^*) & \bar\partial\partial A\\
\bar\partial\partial A^* & 0 \end{array}
\right), \mbox{ at }p.
$$
Moreover
%,
%$$
%\bar\partial G^{-1}=G^{-1}\cdot \bar\partial G \cdot G^{-1}
%$$
%and taking into account
%$$
%A\cdot A^*=0  \quad \mbox{ at }p\,,
%$$
we have
$$
\begin{aligned}
\partial G\wedge\left( G^{-1}\cdot  \bar \partial G\cdot G^{-1}\right) &=
\left(
\begin{array}{cc}
\partial G_B & \partial A\\
\partial A^* & 0\end{array}
\right)\wedge
%\left(
%\begin{array}{cc}
%G_B^{-1} & 0\\
%0 & 1 \end{array}
%\right)\cdot
\left(
\begin{array}{cc}
\bar\partial G_B & \bar\partial A\\
\bar \partial A^* & 0\end{array}
\right)
%\cdot
%\left(
%\begin{array}{cc}
%G_B^{-1} & 0\\
%0 & 1 \end{array}
%\right)
\\
%&=\left(
%\begin{array}{cc}
%\partial G_B& \partial A\\
%\partial A^* & 0\end{array}
%\right)\wedge
%\left(
%\begin{array}{cc}
%\bar\partial\,G_B  & \bar\partial A\\
%\bar \partial A^*  & 0\end{array}
%\right)\\
&=
\left(
\begin{array}{cc}
\partial G_B\wedge \bar\partial \,G_B +\partial A\wedge\bar \partial A^* &\partial G_B\wedge \bar\partial A \\
\partial A^*\wedge \bar\partial G_B  & \partial A^*\wedge \bar\partial A\end{array}
\right)
\end{aligned}
$$
at $p$.
It follows that
$$
\begin{aligned}
R_{1\bar 1}=& \bar\partial\partial(G_B+A\cdot A^*)+\partial G_B\wedge \bar\partial G_B+\partial A\wedge\bar \partial A^*\\
=&\tilde R_B+\bar \partial A\wedge \partial A^*
\end{aligned}
$$
at $p$, and
$$
\begin{aligned}
R_{1\bar 2}&=\bar\partial \partial A+\partial G_B\wedge \bar\partial A\\
R_{2\bar 1}&=\bar\partial \partial A^*+\partial A^*\wedge \bar\partial G_B=\bar\partial(\partial A^*\cdot G_B^{-1})  \\
R_{2\bar 2}&=\partial A^*\wedge \bar\partial A\,,
\end{aligned}
$$
at $p$.
Moreover
$$
\tilde R\wedge \tilde R=\left(
\begin{array}{cc}
R_{1\bar 1}\wedge R_{1\bar 1}+R_{1\bar 2}\wedge R_{ 2\bar 1} & *\\
* & R_{2\bar 2}\wedge R_{2\bar 2}+R_{2\bar 1}\wedge R_{1\bar 2}\end{array}
\right)
$$
and
$$
{\rm tr} (\tilde R\wedge \tilde R) ={\rm tr}(R_{1\bar 1}\wedge R_{1\bar 1}+R_{1\bar 2}\wedge R_{ 2\bar 1})+R_{2\bar 2}\wedge R_{2\bar 2}+R_{2\bar 1}\wedge R_{1\bar 2}\,.
$$
We have
$$
{\rm tr}\,(R_{1\bar 1}\wedge R_{1\bar 1})=
{\rm tr}\left(\tilde R_B\wedge \tilde R_B+2\tilde R_B\wedge\bar \partial A\wedge \partial A^*+\bar \partial A\wedge \partial A^*\wedge \bar \partial A\wedge \partial A^*\right)
$$
at $p$ and a direct computation yields
$$
{\rm tr}\,(\bar \partial A\wedge \partial A^*\wedge \bar \partial A\wedge \partial A^*)=-R_{2\bar 2}\wedge R_{2\bar 2}
$$
at $p$, and so
$$
\begin{aligned}
{\rm tr}(\tilde R\wedge \tilde R)=&\,{\rm tr}\big(\tilde R_B\wedge \tilde R_B+2\tilde R_B\wedge\bar \partial A\wedge \partial A^*\\
&\,+2\,\bar\partial \partial A\wedge \bar\partial(\partial A^*\cdot G_B^{-1})+2\,\partial G_B\wedge \bar\partial A\wedge \bar\partial(\partial A^*\cdot G_B^{-1})\big)
\end{aligned}
$$
at $p$.
We further have
$$
\begin{aligned}
\partial\bar\partial(\bar\partial A\wedge \partial A^*\cdot G_B^{-1}) =\,&\bar\partial\partial A\wedge \bar\partial(\partial A^*\cdot G_B^{-1})+\bar\partial A\wedge \partial\bar\partial(\partial A^*\cdot G_B^{-1})\\
\end{aligned}
$$
and
$$
\begin{aligned}
\bar\partial A\wedge \partial\bar\partial(\partial A^*\cdot G_B^{-1})=\,& \bar\partial A\wedge \bar\partial(\partial A^*\wedge \partial G_B^{-1})=
-\bar\partial A\wedge \bar\partial(\partial A^*\wedge G_B^{-1}\cdot \partial G_B\cdot G_B^{-1})\\
=\,&-\bar\partial A\wedge \bar\partial(\partial A^*\cdot G_B^{-1})\wedge  \partial G_B+\bar\partial A\wedge \partial A^*\wedge  \bar\partial(\partial G_B\cdot G_B^{-1})
\end{aligned}
$$
at $p$,
which together imply
$$
\begin{aligned}
&{\rm tr}(\partial\bar\partial(\bar\partial A\wedge \partial A^*\cdot G_B^{-1}))=\\
&={\rm tr}\left( \tilde R_B\wedge \bar \partial A\wedge \partial A^*+\bar\partial \partial A\wedge \bar\partial(\partial A^*\cdot G_B^{-1})+\partial G_B\wedge\bar\partial A\wedge \bar\partial(\partial A^*\cdot G^{-1}_B)\right)\,.
\end{aligned}
$$
at $p$.
Therefore
$$
{\rm tr}(\tilde R\wedge \tilde R)={\rm tr} \left(\tilde R_{B}\wedge \tilde R_B+2  \partial\bar\partial(\bar\partial A\wedge \partial A^*\cdot G_B^{-1})\right)\,,\mbox{ at }p\,.
$$
This last expression is in fact local since $\rho_{\tilde g_B}=-i{\rm tr}(\bar\partial A\wedge \partial A^*\cdot G_B^{-1})$ is a transverse form on $M$ (this is a standard computation in coordinates).

The  last formula can be in particular applied to $$g_u:={\rm e}^u\,g_B+\tfrac{i}{2}\,\theta\wedge\bar\theta.$$ Since
$R_u=\partial\bar\partial u\,{\rm I}+R,$ the Chern curvature of ${\rm e}^u g_B$ is $\partial\bar\partial u\,{\rm I}+R_B$ and $\rho_{{\rm e}^ug_B}={\rm e}^{-u}\rho_{g_B}$, we get
$$
{\rm tr} (R_u\wedge  R_u)={\rm tr}( R_B\wedge R_B +2\partial\bar\partial u\,{\rm I}\wedge R_B +2\partial \bar \partial ({\rm e}^{-u}\,\rho_{g_B}))+2\partial\bar\partial  u\wedge \partial\bar\partial u\,.
$$

Finally we observe that our assumptions imply that the pair $(\omega_B,\psi_B)$ induces a transverse Calabi-Yau structure on $M$. Consequently ${\rm tr}R_{B}=0$ (see e.g. \cite{HV}) and we obtain ${\rm tr}\,\partial\bar\partial u\,{\rm I}\wedge R_B=0$ and the claim follows by setting $\rho=\rho_{g_B}$.
\end{proof}

\begin{proof}[Proof of Theorem $\ref{main}$] Equation \eqref{pbpou} and Lemma \ref{RuRu} imply that \eqref{ppo} can be written as
\begin{equation}\label{equation}
\begin{aligned}
i\partial \bar\partial {\rm e}^u\wedge\omega_B- \frac{\alpha'}{2} i\partial \bar\partial ({\rm e}^{-u}\rho)-\frac{\alpha'}{2}\partial\bar\partial u\wedge \partial\bar\partial u
+\mu' \frac{\omega_B^2}{2}=0
\end{aligned}
\end{equation}
where $u$ is an unknown basic function and $\rho$ and $\mu'$ are  a given  basic real $(1,1)$-form and a given function, respectively. Taking into account Lemma \ref{RuRu}, equation \eqref{f2} implies that the integral of $\mu'$ is zero.

The solvability of  equation \eqref{equation} can be shown by adapting the proof in \cite{FuYau,picard} to the foliated case. Following the approach in \cite{FuYau}, for $t\in[0,1]$ and $u\in C^{2}_B(M)$ we set
$$
\omega_{u,t}:={\rm e}^{u}\omega_{B}+t\alpha {\rm e}^{-u}\rho+2\alpha i\partial \bar\partial u
$$
and
\begin{equation*}
\tag{ $*_t$}
i\partial \bar\partial {\rm e}^u\wedge\omega_B-it \frac{\alpha'}{2} i\partial \bar\partial ({\rm e}^{-u}\rho)-\frac{\alpha'}{2}\partial\bar\partial u\wedge \partial\bar\partial u
+t\mu' \frac{\omega_B^2}{2}=0\,.
\end{equation*}

Moreover, for $A>0$ and $0<\beta<1$ let
$$
\Upsilon_A=\{u\in C^{4,\beta}_B(M)\,\,:\,\, \|{\rm e}^{-4u}\|_B=2A^4  \}\,,
$$
where $(\cdot,\cdot)_B$ is the basic scalar product defined for $\eta,\gamma\in \Omega^r_B(M)$ as
$$
(\eta,\gamma)_B=\int_M\eta\wedge*_B\gamma\wedge \chi\,,
$$
and $C^{r,\beta}_B(M)$ is the space of H\"older functions on $M$ which are also basic. Then for $t\in[0,1]$ we consider
$$
\Upsilon_{A,t}=\{u\in \Upsilon_A\,\,:\,\, \omega_{u,t}+\chi>0  \}
$$
and
$$
{\bf T}=\left\{t\in [0,1]\,\,:\,\, \mbox{ there exists }u\in \Upsilon_{A,t} \mbox{ solving }(*_t) \right\}\,.
$$
The proof consists in showing that ${\bf T}$ is open and closed in $[0,1]$.

\medskip
In order to show that ${\bf T}$ is open, let $t_0\in {\bf T}$ and let $u_{t_0}\in \Upsilon_{A,t_0}$.  Then we consider
$$
\Upsilon_{[0,1]}=\{(t,u)\in [0,1] \times \Upsilon_A \,\,:\,\,u\,\,\in \Upsilon_{A,t}\}\,.
$$
and we define the operator $L\colon \Upsilon_{[0,1]}\to C^{2,\beta}_{B,0}(M)$ as follows
$$
L(t,u)=*_B\left(i\partial \bar\partial {\rm e}^u\wedge\omega_B-it\frac{\alpha'}{2} \partial \bar\partial ({\rm e}^{-u}\rho)-\frac{\alpha'}{2}\partial\bar\partial u\wedge \partial\bar\partial u
+t\mu' \frac{\omega_B^2}{2}\right)\,.
$$
We have
$$
dL[t_0,u_0](s,\varphi)=L'(\varphi)
$$
for every $\varphi\in T_{u_{t_0}}\Upsilon_{A}$, where for $\eta\in C^{2}_B(M)$.
$$
L'(\eta)=*_B\left(i\partial \bar \partial({\rm e}^{u_{t_0}}\eta)\wedge\omega_B+it_0\tfrac{\alpha'}{2}\partial \bar \partial({\rm e}^{-u_{t_0}}\eta\rho)-i\alpha'\partial u_{t_0}\wedge\partial \bar\partial \eta\right)\,.
$$

Condition $\omega_{u_0,t_0}+\chi>0$ implies that $L'$ is transversally elliptic.  The same computation as in \cite{FuYau} implies that the formal adjoint of $L'$ with respect to $(\cdot,\cdot)_B$ is the complex  Laplacian of $\omega_{u_0,t_0}+\chi$ restricted to basic functions. Denote this operator by $P$.

We can now show that $L'\colon T_{u_{t_0}}\Upsilon_{A}\to  C^{2,\beta}_{B,0}(M)$ is an isomorphism. If $\eta\in C^{2}_B(M)$ belongs to the kernel of $L'$, then  it belongs to the kernel of $L'+\frac{i}{2}Z\bar Z$ which is a genuine elliptic operator. Hence $\eta$ is constant and, imposing $\eta\in T_{u_{t_0}}\Upsilon_{A}$, then we get $\eta=0$. Therefore $L'$ is injective. In view of El Kacimi theorem \cite{EKA}, a basic function $\psi\in C^{0,\beta}_{B,0}(M)$ is in the image of $L'$ if and only  if it is orthogonal to the kernel of $P$. Again we can use that $P+\tfrac{i}{2}Z\bar Z$ is a genuine elliptic operator and we deduce that  $P\colon  C^{2,\beta}_{B,0}(M)\to C^{0,\beta}_{B,0}(M)$ is injective. The surjectivity of $L'$ then follows.  Therefore the inverse function theorem implies that ${\bf T}$ is open.
The closedness of ${\bf T}$ can be obtained by showing the same a priori estimates as in \cite{FuYau}. Indeed, the same computations as in \cite{FuYau} can be done in our case using transverse holomorphic coordinates and replacing the $L^2$-product with the $(\cdot,\cdot)_B$-product on  basic forms. Since this part is a straightforward adaptation of  \cite[Sections 8--12]{FuYau} we just give a rough description of how things work pointing out the main steps.

Let $(u,t)\in\Upsilon_{[0,1]}$ and let $P$ be the complex Laplacian operator of $\omega_{u,t}$ restricted to basic functions. By integrating
$$
P({\rm e}^{-ku})\,\omega_{u,t}\wedge \chi
$$
on $M$ for $k\in\mathbb N$, using equation $(*_t)$ and following the computations in \cite[Section 12]{FuYau}, we deduce
$$
k\int_M {\rm e}^{-(k-1)u}|\nabla u|^2{\rm Vol'}\leq C\int_M {\rm e}^{-(k+1)u}{\rm Vol'}+
C\int_M {\rm e}^{-ku}{\rm Vol'}
$$
where ${\rm Vol'}=\tfrac12 \omega_{u,t}^2\wedge \chi$ and $C$ is a positive constant depending on $\alpha',\rho,\omega_B,\mu$, only. By replacing $k$ with $k+1$
and using Sobolev inequality it follows
$$
\|{\rm e}^{-ku/2}\|_{L^4}\leq C\left(\|{\rm e}^{-ku/2}\|_{L^2}+\|\nabla {\rm e}^{-ku/2}\|_{L^2}\right)\,,
$$
where $L^p$-norms are computed with respect to ${\rm Vol}'$ 
Using H\"older inequality and integrating by parts from this last inequality it follows a $C^0$ a priori estimate on $u$ for a suitable choice of $A$ \cite[Proposition 21]{FuYau}.
By using the equation it follows an upper bound for $|\nabla u|$ \cite[First part of section 9]{FuYau}. The next step consists in proving an upper bounds of
$$
F=\frac{\omega_{u,t}^2}{\omega_B^2}\,.
$$
Notice that $F$ is a basic function since both $\omega_{u,t}$ and  $\omega_B$ are basic forms. Accordingly to \cite[Proposition 22]{FuYau} the bound on $F$ is obtained by applying the maximum principle to the function
$$
G_1:=1-\alpha'{\rm e}^{-u}|\nabla u|^2+\alpha'{\rm e}^{-\epsilon u}-2\alpha'{\rm e}^{-\epsilon \inf u}
$$
for a suitable choice of $\epsilon$.  Following \cite{FuYau} it is enough to prove the higher order estimates in the special case $\rho=-f\omega_B$ for a positive basic function $f$ and then deduce the general case by observing that the term ${\rm e}^u$ can always control terms such as ${\rm e}^{-u}|\rho|$.
The estimate on the second derivatives of $u$ is then obtained by applying the maximum principle to
$$
G_2={\rm e}^{-\lambda_1u+\lambda_2|\nabla u|^2}\left({\rm e}^u+f{\rm e}^{-u}+\frac{\alpha'}{2}\Delta_B u\right)
$$
for a suitable choice of positive constants $\lambda_1$ and $\lambda_2$ \cite[Proposition 23]{FuYau}. Here $\Delta_B$ denotes the basic Laplacian operator with respect to the metric induced by $\omega_B$. About the third order estimates on $u$ let
$$
G_3=
\left(\lambda_3+\frac{\alpha'}{2}\Delta_B
u\right)\Theta+\lambda_4\left(m+\frac{\alpha'}{2}\Delta_B
u\right)\Gamma+\lambda_5|\nabla
u|^2\Gamma+\lambda_6\Gamma,
$$
where $\lambda_3, \lambda_4,\lambda_5,\lambda_6,m$  are positive constants, $m$ satisfies
$$
m+\frac{\alpha'}{2}\Delta_B u>0
$$
and
\begin{eqnarray*}
&&\Gamma=g^{i\bar j}g^{k\bar l}u_{,ik}u_{,\bar j\bar l}\\
&&\Theta=g'^{i\bar r}g'^{s\bar j}g'^{k\bar t}u_{,i\bar jk}u_{,\bar
r s\bar
t}\\
&&\Xi:=g'^{i\bar j}g'^{k\bar l}g'^{p\bar q}u_{,ikp}u_{,\bar j\bar l
\bar  q}\,,\\
&&\Phi:=g'^{i\bar j}g'^{k\bar l}g'^{p\bar q}g'^{r\bar s}u_{,i\bar l
pr}u_{,\bar jk\bar q
\bar  s}\,,\\
&&\Psi=g'^{i\bar j}g'^{k\bar l}g'^{p\bar q}g'^{r\bar s}u_{,i\bar l
p\bar s}u_{,\bar jk\bar qr}\,.
\end{eqnarray*}
Here we are using the following notation: $\{z^1,z^2\}$ are local complex transverse coordinates, $\omega_B=g_{i\bar j}dz^i\wedge d\bar z^j$, $\omega_{u,t}=g'_{i\bar j}dz^i\wedge d\bar z^j$ and indices preceded by a comma indicate covariant differentiation with respect to the transverse Levi-Civita connection of the metric induced by $\omega_B$. Following \cite[Section 11]{FuYau} the computation of $P(G_3)$ at a maximum point $q$ of $G_3$
leads to the following inequality
$$
0\geq P(G_3) \geq\frac{m_2}{4}\Theta^2+m_2\Theta\Gamma+
m_1\Gamma^2+m_3\Xi-C_4\Theta-C_4\Gamma-C_4 \quad\mbox{ at }q
$$
for a suitable choice of the constants $\lambda_i$ and $m$ which implies a uniform $C^3$ estimate on $u$.
As in the proof of Calabi-Yau theorem, the $C^3$ estimate  is enough to deduce higher order estimates via bootstrapping argument. This implies the closure of ${\bf T}$, as required.

\end{proof}

\begin{oss}
Equation \eqref{equation} is equivalent to the following system on $M$
$$
\begin{cases}
i\partial \bar\partial {\rm e}^u\wedge\omega^2-\alpha i\partial \bar\partial ({\rm e}^{-u}\rho)\wedge \omega-\alpha\partial\bar\partial u\wedge \partial\bar\partial u\wedge\omega
+\mu \frac{\omega^3}{6}=0\,,\\\
Z(u)=0\,,
\end{cases}
$$
where the unknown function $u$ belongs to $C^{\infty}(M,\R)$ and $\alpha$ and $\mu$ depend on $\alpha'$ and $\mu'$ in a universal way. The first equation is analogue to the equation considered in \cite{FuYau,{picard}}, but in our case $\omega$ is not K\"ahler. Hence a generalization of the main theorem in \cite{picard} to the  general non-K\"ahler setting should lead to a direct proof of the solvability of \eqref{equation} without
adapting Fu-Yau proof to the foliated case. 
\end{oss}

\medskip
Now we are ready to prove Theorem A.

\begin{proof}[Proof of Theorem A] Let $(X,\omega_X)$, $M$, $\omega_1$, $\omega_2$ and $E$ as in the statement of Theorem A. Then Theorem \ref{toric} implies that $M$ has a complex structure which makes $\pi$ holomorphic and
$\pi^*(\omega_1+i\omega_2)=\bar\partial \theta$, where $\theta=\theta_1+i\theta_2$ and $\theta_1$ and $\theta_2$ are the connections $1$-forms.
By setting $\omega_B=\pi^*(\omega_X)$ and $\psi_B=\pi^*(\psi_X)$,  where $\psi_X$ is a non-vanishing holomorphic $(2,0)$-form on $X$, we have that $(\theta, \omega_B,\psi_B)$ satisfies the assumptions \eqref{top} in the statement  of Theorem \ref{main}. Moreover, since $E$ is a stable vector bundle of degree $0$ over $(X,\omega_X)$, then it has a Hermitian metric $h$ along its fibers such that, if $V=\pi^*(E)$ and $H=\pi^*(h)$, then the curvature $F_H$ of $H$ satisfies \eqref{f1}.
Finally  \eqref{f2} in the assumptions of Theorem \ref{main} can be written as
\begin{multline*}
\int_M \alpha'{\rm tr}\,(R\wedge R)\wedge \omega-\alpha'{\rm tr}\,(F_H\wedge F_H)\wedge \omega- \|d\theta\|^2\,\omega^3 = \\
\frac{i}{2} \int_M \alpha'{\rm tr}\,(R\wedge R)\wedge \theta\wedge\overline{\theta}-\alpha'{\rm tr}\,(F_H\wedge F_H)\wedge \theta\wedge\overline{\theta}- \|d\theta\|^2\,\omega_X^2\wedge\theta\wedge\overline{\theta}  =\\
\frac{i}{2}\int_M (\alpha'{\rm tr}\,(R\wedge R)-\alpha'{\rm tr}\,(F_H\wedge F_H)- (\|\omega_1\|^2+\|\omega_2\|)\,\omega_X^2)\wedge\theta\wedge\overline{\theta}  = 0\,
\end{multline*}
which is equivalent by the Fubini's Theorem to \eqref{onE} in Theorem A  since the first two terms are the characteristic classes on the base orbifold by the orbifold Chern-Weil theory.
Hence all the assumptions in Theorem \ref{main} are satisfied and the claim follows.
\end{proof}

\begin{oss}  In the examples we consider in the next Section, an explicit and simple formula for $e(S)$ is given in \cite[Sect. 7.3,  p. 115]{IF}. For a K3 orbifold $S$  with $k$ isolated  $A_1$-singularities,  the formula is $e(S)= 24 - k$.  Note that  every K3 orbifold admits a K\"ahler Ricci flat metric, see e.g. \cite{Faulk}.

\end{oss}

\section{Proof of Theorem  B}

Recall from Section 2  that  in general for  a   principal $T^2$-bundle  $\pi: M\rightarrow X$ over  a K\"ahler  manifold $(X, \omega)$, with connection forms $\theta_1, \theta_2$,  the natural Hermitian structure on $M$ has a fundamental form given by  {$ F = \theta_1\wedge \theta_2+ \pi^*(\omega_X)$.

Also for the Fu-Yau construction we need to consider $\omega = \theta_1\wedge \theta_2 + \pi^*(e^f \omega_X )$.
 In particular $\omega$ is conformally balanced, if the curvature forms are primitive with respect to $\omega_X$. Since all of the calculations are local, the same is valid if we replace $(X,\omega_B)$ by a K\"ahler orbifold, so $M$ becomes a $T^2$-bundle in the terminology of \cite{HS} (see Section \ref{sect2}).

For the examples below we also need $X$ to be a K3 orbifold  with only cyclic singularities, as in the Reid's  \cite{Reid} and Iano-Fletcher's lists \cite{IF}. We want  to find Seifert $S^1$-bundles $Y$ over $X$ with smooth total space and primitive characteristic class $c_1(Y/X)$.

\begin{lemma}
All examples in Reid's and Iano-Fletcher's lists have trivial orbifold fundamental group  and admit an ample rational divisor.
\end{lemma}

%%%%%%%%%%%%%%%%%%%%%%%%%%%%%%%%%%%%%%
\begin{proof}
The divisor comes from the embedding in a  weighted projective space and the hyperplane section, the triviality of the orbifold fundamental group is  Theorem 4.7.12 in  \cite{BG}.
\end{proof}

%%%%%%%%%%%%%%%%%%%%%%%%%%%%%%%%%%%%%%%%%%%%

In order to complete the proof of Theorem B, we need also to assure the existence of a stable vector bundle on specific surfaces from that lists. Consider the example under number  14 at  page 143  in the  Iano-Fletcher's  list of codimension two K3 orbifolds.   It is an intersection of two degree 6 hypersurfaces  in ${\mathbb P}(2,2,2,3,3)$. It has 9 isolated  $A_1$-singularities. Then blow up this K3 orbifold  at  $9-k$ points, where $1\leq k \leq 8$ and denote the resulting surface by $X_k$. Since the blow-ups lead to exceptional divisors with negative intersections, they could be used to construct primitive classes. For $X_k$ we have:

\begin{lemma}\label{stable}
Let $X_k$ be the $K3$-orbifold surface obtained by blowing-up $k, 0\leq k \leq 9$ singular points of the general intersection of two hypersurfaces  of degree $6$ in $\mathbb{P}(2,2,2,3,3)$. Then for $k>0$ there exists  on $X_k$ a stable bundle $E$ of rank 2 and with vanishing first Chern class  and  $c_2(E) = c$ for any $c\geq 4+\frac{k}{2}$.
\end{lemma}

\begin{proof}  We follow closely Theorem 5.1.3 from \cite{HuyLehn10}. The existence of such bundle is based on the Serre construction. The main observation is that the construction depends on a choice of a   0-dimensional subscheme  (isolated points which could be chosen different from the singular ones). We denote by $H$ the hyperplane ample divisor on $X_0$ as well as its pull-back to the blow-ups $X_k$.  We first can construct a bundle $E'$ with $det(E') = 2H$ and $c_2(E') = c+1$, and then the bundle $E = E'\otimes\mathcal{O}_{X_k}(-H)$ will be stable with $c_1(E)=0$ and $c_2(E) = c_2(E') -H^2 = c$.

Now we construct $E'$ from the sequence:
$$
0\rightarrow \mathcal{O}_{X_k} \rightarrow E' \rightarrow \mathcal{Q}\otimes\mathcal{I}_Z \rightarrow 0
$$
where the support of $Z$ is a set of points different from the singular ones. The existence of such bundle $E'$ follows in the same way as the existence part of the proof of Theorem 5.1.1 in \cite{HuyLehn10}, since it is based on the local arguments around the points in the support of $Z$.  Then $c_2(E)=l(Z)$, the number of points in $Z$. Take $l_1 = h^0(K_{X_k}\otimes\mathcal{O}(2H)) = h^0(\mathcal{O}(2H))$.

Since $2H$ is big and nef on $X_k$, we have by Kowamata-Viehweg vanishing Theorem (which is valid for the orbifolds $X_k$)  $l_1 = \chi(2H)$ and,  by the Kawasaki-Hirzebruch-Rieamnn-Roch Theorem,  $\chi(2H) = \chi_{orb}(2H)+ \frac{9-k}{2} = 4+\frac{9-k}{2}$. Now if $l>l_1$, the pair $(K_{X_k}\otimes\mathcal{O}(2H), Z)$ has the Cayley-Bacharach property. Hence by the Serre construction there exists an extension as above.

The proof that the bundle $E$ is stable follows precisely the proof in Theorem 5.1.3 from \cite{HuyLehn10}. Following the same notations for $l_2$ we can see that $l_2 = 0$,  since $l_2$ is the dimension of effective divisors of degree $d=0$. This shows that the bundle $E'$ is stable by the same argument.

\end{proof}

%%%%%%%%%%%%%%%%%%%%%%%%%%%%%%%%%%%%%%%%%%%%

Now we can prove Theorem B

\begin{proof}[Proof of Theorem B] Suppose that $X_k$ is the surface from above. If $H$ is the restriction of the hyperplane divisor in ${\mathbb P}(2,2,2,3,3)$ to $X_0$, it is positive and every blown-up point defines a divisor $E$ in the blow-up with $E.E=-2$. In particular we can apply  the Theorem \ref{cororb}  above $9-k$ times. This will provide a smooth $\tilde{M}$ equipped with a smooth non-negative $2-$form $\omega$. Now Theorem \ref{toric} above for $M = S^1\times \tilde{M}$ and $d\theta_1 = 0$ for $\theta_1$ being the $S^1$-volume form will provide a conformally balanced metric and a unitary $(3,0)$-form, and together with the Lemma \ref{stable} above provides a Hermitian-Yang-Mills instanton bundle $E$ which satisfies the condition (5) in Theorem A for a positive $\alpha'$.

 % An example of such bundle for mildly singular K3 surfaces is provided by their tangent sheaf according to \cite[Theorem A]{Gue}.

   By Theorem \ref{main}  it will admit a solution of the Hull-Strominger system. The diffeomorphism type of $M$ is $S^1\times \sharp_k(S^2\times S^3)$ for appropriate $k$, which follows from Barden's results \cite{Barden} and Theorem \ref{Kol} for simply-connected 5-manifolds with a semi-free $S^1$-action (just as in the well-studied Sasakian case).

  We also need to  determine  the orbifold second Betti number of the surface, in order to  find $k$. The calculation follows from  Theorem 3.2 in \cite{cuadros}. We can see that there is no torsion in $H_2(M, {\mathbb Z})$ by Theorem 3.4 in \cite{cuadros} (which is a theorem of J. Koll\'ar).  Since for the smooth K3 surface we have  $b_2 = 22$, then for the singular one it should be $22-9=13$ and then we have a smooth Seifert $S^1$-bundle $M$ which is diffeomorphic to $\sharp_k(S^2\times S^3)$ for $k$  as required. This provides the existence of solutions of the Hull-Strominger system on $S^1\times\sharp_k(S^2\times S^3)$,  for $13\leq k\leq 22$. For the simply-connected case the proof is the same, but instead of Theorem \ref{Kol}, we use Corollary \ref{cormain}. Since we need two independent divisors $D_1$ and $D_2$ - we need to have at least 2 points blown-up. So  in this case we have solutions on $M = \sharp_r (S^2\times S^4)\sharp_{r+1}(S^3\times S^3)$,  for $14\leq r \leq 22$.

  %%%%%%%%%%%%%%%%

\end{proof}

\begin{oss}
In \cite{Lee}  solutions for the Strominger system, satisfying the anomaly cancelation condition (3), but with a weaker version of (4), are constructed on $\sharp_r (S^2\times S^4)\sharp_{r+1}(S^3\times S^3)$ for every $r>0$. The underlying complex structure does not have holomorphically trivial canonical  bundle, but has vanishing first Chern class and admits a CYT metric. The metric and complex structure are constructed in \cite{GGP}, and \cite{Lee} uses this construction to find solutions for (3). We note that on $(S^2\times S^4)\sharp_2(S^3\times S^3)$, J. Fine and D. Panov \cite{FP} constructed a different complex structure with $\mathbb{C}^*$-action and holomorphically trivial canonical bundle. Although it is unclear whether their example admits a solution to the Hull-Strominger system, Theorem \ref{main} may provide a possible approach.

\end{oss}

\begin{oss}
 Finally we mention for completeness a partial converse of the construction in Section \ref{sect2} for spaces carrying solutions of the Hull-Strominger system with $T^2$ symmetry. Let $M$ be a  compact complex manifold of complex dimension $3$ with trivial canonical bundle admitting a balanced metric  (so either  K\"ahler or non-K\"ahler Calabi-Yau). Suppose that $M$ admits a locally free  $T^2$-action
 preserving the metric and the complex structure and with only finitely many holomorphic orbits having non-trivial isotropy. If the orbits define a Hermitian foliation, i.e.  a transversally holomorphic and Riemannian foliation, then the leaf space is a  compact complex orbifold surface with an induced Hermitian metric
 \cite{molino, Sund}.

  If additionally we assume that there are only finitely many leaves,  then we have an orbifold $T^2$-bundle $\pi:M\rightarrow S$ where $S$ has only isolated cyclic singularities, since the local holonomy groups are finite and abelian. Indeed, by \cite{Epstein} any leaf of a  compact foliation has finite holonomy group if and only the leaf space  is Hausdorff and for a transversally Riemannian  foliation  the leaf  space  is Hausdorff.

   Moreover, by \cite[Prop. 3.7, page 94]{molino} $S$ has a Hermitian metric and $\pi$ is a Hermitian submersion. If $Z$ is a holomorphic (1,0)-vector field induced by the locally free  $T^2$-action, then for the holomorphic (3,0)-form $\psi$ on $M$,  $i_Z\psi$ induces a non-vanishing (2,0)-form on $S$, holomorphic in the orbifold sense. If we assume that $M$ is simply-connected, then $\pi_1^{orb}(S)$ is trivial. After resolving the singularities of $S$,  we obtain a smooth simply-connected $\overline{S}$ with trivial canonical bundle. So $\overline{S}$ is a K3 surface, according to Kodaira classification of compact complex surfaces. Then it follows that when $M$ has torsionless cohomology, it is diffeomorphic to $\sharp k(S^2\times S^4)\sharp (k+1)(S^3\times S^3)$ for $1\leq k \leq 22$,  where $k$ depends on the number of blow-ups in the resolution of $S$. If $S$ itself is smooth, then  $k=22$.
\end{oss}

\noindent {\bf Acknowledgments.} This project has started and later significantly developed during two visits of one of us (GG) to Torino. The second named author wishes to thank the Mathematics Department of Universit\`{a} di Torino for the warm hospitality.  We are grateful  to Tony Pantev  for the discussions and help on Lemma 4.2 and its proof.  We would also like to   thank   Ugo Bruzzo and Cinzia Casagrande for useful comments.
 The authors wish also  to thank the anonymous referee for  careful reading of the manuscript  and for pointing out a gap in the previous version of the paper as well as the helpful suggestions.


\begin{thebibliography}{12}%%%%%%%%%%%%%%% BIBLIOGRAFIA%%%%%%%%%%%%%%%%%%%%%%%

\bibitem{AGG} K. Abe, D. Grantcharov, G. Grantcharov, {\em
On some complex manifolds with torus symmetry},  Recent advances in Riemannian and Lorentzian geometries (Baltimore, MD, 2003), 1–8,
Contemp. Math., {\bf 337}, Amer. Math. Soc., Providence, RI, 2003.


\bibitem{Barden} D. Barden, Simply connected five manifolds, {\em Ann. of Math.} (2) {\bf 82} (1965), 365--382.

\bibitem{BS}  I. Biswas, G. Schumacher,  The Weil-Peterson currect for moduli of vector bundles and applications to orbifolds,  {\em Ann. Fac. Sci. Toulouse Math.}  (6) {\bf 25} (2016), no. 4, 895--917.

\bibitem{BG} C. P.  Boyer,  K. Galicki,  {\em Sasakian geometry},  Oxford Mathematical Monographs, Oxford University Press, Oxford, 2008.

\bibitem{CE} E. Calabi and B. Eckmann, A class of compact, complex manifolds which are not algebraic,  {\em Ann. of  Math.}  {\bf 58} (1953), 494--500.

\bibitem{CP} F. Campana, T. Peternell,
Algebraicity of the ample cone of projective varieties,
 {\em J. Reine Angew. Math.} {\bf 407} (1990), 160--166.

\bibitem{CHSW} P. Candelas, G.T. Horowitz, A. Strominger, E.  Witten,  Vacuum configurations for superstrings, {\em Nucl.
Phys. B} {\bf 258} (1985), 46--74.

\bibitem{C1}  J. Chu, L. Huang, X. Zhu,
 The Fu-Yau equation on compact astheno-K\"ahler manifolds, {\em  Adv. Math.} {\bf 346} (2019), 908–945.

\bibitem{C2} J. Chu, L. Huang, X. Zhu, The Fu-Yau equation in higher dimensions, {\em Peking Math. J.} {\bf 2} (2019), no. 1, 71–97.

\bibitem{cuadros} J. Cuadros,  Null Sasaki eta-Einstein Structures in Five Manifolds,   {\em Geom. Dedicata} {\bf 169} (2014), 343--359.


\bibitem{D-S} X. de la Ossa, E.E. Svanes,  Connections, field redefinitions and heterotic supergravity, {\em J. High Energ. Phys.} 2014, no. 10, 123, front matter+54 pp.


\bibitem{Donaldson} S. K. Donaldson,  Anti self-dual Yang-Mills connections over complex algebraic
surfaces and stable vector bundles,  {\em Proc. London Math. Soc.}  {\bf 50}  (1985), 1--26.

\bibitem{EKA}
A. El Kacimi-Alaoui, Op\'erateurs transversalement elliptiques sur
un feuilletage riemannien et applications, \emph{Compositio Math.}
{\bf 73} (1990), 57--106.

\bibitem{Epstein} D. Epstein, Foliations with all leaves compact, {\em Ann. Inst. Fourier}  {\bf 26} (1976), no. 1,   265--282.

\bibitem{Faulk} M. Faulk, On Yau's theorem for effective orbifolds, {\em Expo. Math,}  {\bf 37}, n. 4, (2019), 382--409.

\bibitem{Fei18}  T. Fei, Generalized Calabi-Gray Geometry and Heterotic Superstrings, {\tt arXiv:1807.08737}.


\bibitem{FHP}  T. Fei, Z.-J. Huang,  S. Picard, A construction of infinitely many solutions to the Strominger system,  {\tt arXiv:1703.10067}, To appear in {\em J. Diff. Geom.}

\bibitem{FHP17} T. Fei, Z.-J. Huang,  S. Picard,  The Anomaly flow over Riemann surfaces, {\tt arXiv:1711.08186} to appear in {\em Int. Math. Res. Notices}.

 \bibitem{FY14} T. Fei , S.-T. Yau, Invariant Solutions to the Strominger System on Complex Lie
Groups and Their Quotients, {\em Comm. Math. Phys.} {\bf 338} (2015) No. 3, 1183-1195.

\bibitem{FIUV} M. Fern\'andez, S. Ivanov, L. Ugarte, R. Villacampa, Non-K\"ahler heterotic-string compactifications with non-zero fluxes and constant dilaton, { \em Commun. Math. Phys.}  {\bf 288}
(2009),  677--697.

\bibitem{FP} J. Fine, D. Panov, Hyperbolic geometry and non-K\"ahler manifolds with trivial canonical bundle, {\em Geom. Topol.}  {\bf 14} (2010) 1723–-1763.


\bibitem{FuYau2007} J.-X. Fu, S.-T. Yau, A Monge-Amp\`ere type equation motivated by string theory,
{\em Comm. Anal. Geom.} {\bf 15} (2007) 29--76.

\bibitem{FuYau}
J.-X. Fu, S.-T. Yau, The theory of superstring with flux on non-K\"ahler manifolds and the complex Monge-Amp\`ere equation, {\em J. Differential Geom.} {\bf 78} (2008), no. 3, 369--428.



\bibitem{FTY09} J.-X. Fu, L.-S. Tseng and S.-T. Yau, Local heterotic torsional models, {\em Comm. Math.
Phys.} {\bf 289}  (2009) 1151--1169.

 \bibitem{Fernandez16}  M. Garcia-Fernandez, Lectures on the Strominger system, Travaux Math\'ematiques, Special Issue:
School GEOQUANT at the ICMAT, Vol. XXIV (2016) 7--61.

\bibitem{Fernandez18} M. Garcia-Fernandez, T-dual solutions of the Hull-Strominger system on non-K\"ahler threefolds, arXiv: 1810.04740.

\bibitem{GP}
E. Goldstein, S. Prokushkin, Geometric model for complex non-K\"ahler manifolds with ${\rm SU}(3)$ structure,
{\em Comm. Math. Phys.} {\bf 251} (2004), no. 1, 65--78.

\bibitem{GL}
R. Z. Goldstein, L. Lininger, A classification of 6-manifolds with free $S^1$
actions, in \lq \lq Proceedings of the Second Conference on Compact Transformation
Groups"  (Univ. Massachusetts, Amherst, Mass., 1971), Part I, {\em Lecture Notes in
Math.} {\bf 298}, Springer, Berlin, (1972) 316--323.



\bibitem{GGP} D. Grantcharov, G. Grantcharov,  Y.S. Poon, Calabi-Yau connections with torsion on toric bundles, {\em J. Differential Geom} {\bf 78} (2008), no. 1, 13--32.


 %\bibitem{Gue}H. Guenancia, Semistability of the tangent sheaf of singular varieties, {\em Algebraic Geometry}  {\bf 3} (2016), no 5, 508–542.


\bibitem{HV} G. Habib, L. Vezzoni, Some remarks on Calabi-Yau and hyper-K\"ahler foliations,  {\em Differential Geom. Appl.} {\bf 41} (2015), 12--32.

 \bibitem{HS}  A.  Haefliger, E.  Salem, Actions of Tori on Orbifolds,   {\em Ann. Global Anal. Geom.}
 {\bf 9} (1991), no. 1, 37--59.

 \bibitem{Hull86}   C. Hull, Superstring compactifications with torsion and space-time supersymmetry, In
Turin 1985 Proceedings \lq \lq Superunification and Extra Dimensions" (1986),  347--375.


\bibitem{HuyLehn10}   D. Huybrechts, M. Lehn  The Geometry of Moduli Spaces of Sheaves, 2nd ed., {\em Cambridge Univ. Press} (2010).


 \bibitem{IF}  A. R. Iano-Fletcher, Working with Weighted Complete Intersections, Explicit Birational Geometry of $3-$folds, London Math. Soc. Lecture Notes Ser., vol. 281, Cambridge Univ. Press, Cambridge (2000), 101--173.

\bibitem{Kollar} J. Koll\'ar, Shafarevich maps and plurigenera of algebraic varieties, {\em Invent. Math.}
{\bf 113} (1993), no. 1, 177--215.

\bibitem{Kollar2004} J. Koll\'ar, Seifert $G_m$-bundles, {\tt  arXiv:\,math/0404386}.

\bibitem{Kollar2005} J. Koll\'ar,   Einstein metrics on five-dimensional Seifert bundles, {\em J. Geom. Anal.}  {\bf 15} (2005), 445--476.

\bibitem{Kollar2007}  J. Koll\'ar,   Einstein metrics on connected sums of  $S^2 \times S^3$,  {\em J. Differential Geom.}  {\bf 75} (2007), 259--272.

\bibitem{LTX} C. Laurent-Gengoux,  J. L. Tu, P. Xu, Chern-Weil map for principal bundles over groupoids, {\em Math. Z.} {\bf 255} (2007), 451--491.

\bibitem{Lee}H. Lee, Strominger’s System on non-K\"ahler Hermitian Manifolds, PhD Thesis, Oxford University, (2011).

\bibitem{LiYau}
J. Li, S.-T. Yau, Hermitian Yang-Mills connections on non-K\"ahler manifolds, Mathematical aspects of string theory (S.-T. Yau ed.), 560-573, World Scient. Publ. 1987.

\bibitem{LY2005} J. Li, S.-T. Yau, The existence of supersymmetric string theory with torsion,  {\em J. Differential Geom.}  {\bf 70}  (2005) 143--181.


\bibitem{molino} P. Molino, {\em Riemannian foliations}, Progress in Mathematics, {\bf 73}. Birkh\"auser, Boston 1988.

\bibitem{OUV17} A. Otal, L. Ugarte, R. Villacampa, Invariant solutions to the Strominger system and
the heterotic equations of motion, {\em Nuclear Phys. B}   {\bf 920} (2017), 442--474.

\bibitem{PPZ16} D.-H. Phong, S. Picard, X. Zhang,  The Anomaly flow and the Fu-Yau equation,  {\em Ann. PDE}  {\bf 4}  (2018), no. 2, Art. 13, 60 pp.


\bibitem{picard}
D.-H. Phong, S. Picard, X.-W. Zhang,  On estimates for the Fu-Yau generalization of a Strominger system, {\em J. Reine Angew. Math.} {\bf 751} (2019), 243–274.

\bibitem{PPZ16}
D.-H. Phong, S. Picard, X.-W. Zhang, Anomaly flows,    {\em Comm. Anal. Geom.} {\bf 26}  (2018), no. 4, 955--1008.


\bibitem{PPZ17unim} D.-H. Phong, S. Picard, X.-W. Zhang, The Anomaly flow on unimodular Lie groups, {\em Advances in complex geometry,} Contemp. Math., {\bf 735} 217–237, Amer. Math. Soc., Providence, RI, 2019.

\bibitem{PPZ17} D.-H. Phong, S. Picard,   X.-W.  Zhang, The Fu-Yau equation with negative slope parameter, {\em Invent. Math.}  {\bf 209} (2017),
541--576.

\bibitem{PPZ18} D.-H. Phong, S. Picard,   X.-W.  Zhang, Fu-Yau Hessian equations, {\tt arXiv:1801.09842}, to appear in J. Diff. Geom.

\bibitem{PPZ18} D.-H. Phong, S. Picard,   X.-W.   Zhang,  New curvature flows in complex geometry, {\em Surveys in Differential Geometry} {\bf 22} (2017), No. 1, 331-364.

 \bibitem{Reid}  M. Reid,  Young Person's Guide to Canonical Singularities, in Algebraic geometry, Bowdoin, 1985 (Brunswick, Maine, 1985), 345--414,
Proc. Sympos. Pure Math., 46, Part 1, Amer. Math. Soc., Providence, RI, 1987.


\bibitem{Simpson} C. Simpson, Constructing variations of Hodge structure using Yang-Mills theory and applications to uniformization,  {\em J. Amer. Math. Soc.} {\bf 1} (1988), 867--918.

\bibitem{Strominger} A. E. Strominger,  Superstrings with torsion,  {\em Nuclear Phys. B}  {\bf 274}(2) (1986),  253--284.


\bibitem{Sund} D. Sundararaman, {\it Compact Hausdorff Transversally Holomorphic Foliations} , {\em Complex Analysis}  Proceedings of the Summer School. Held at the International Centre for Theoretical Physics, Trieste, July 5 - 30, 1980, eds.  J. Eells, Lect. Notes in Math. {\bf 950} 360--376. 1983.



\bibitem{UY} K. Uhlenbeck, S.-T. Yau, On the existence of Hermitian-Yang-Mills connections in stable vector bundles. Frontiers of the mathematical sciences, {\em Comm. Pure Appl. Math.}  {\bf 39} (1986), 257--293.

\bibitem{Verbitsky} M. Verbitsky, Holomorphic symplectic geometry and orbifold singularities, {\em Asian J.
Math.}  {\bf 4}  (2000), no. 3, 553--563.


\bibitem{Yau} S.-T. Yau, On the Ricci curvature of a compact K\"ahler manifold and the complex Monge-Amp\`ere equation. I., {\em Comm. Pure Appl. Math.}  {\bf 31} (1978), no. 3, 339--411.

\end{thebibliography}
\end{document}